\documentclass[12pt]{article}

\usepackage{mathrsfs}

\usepackage{amsmath,amsthm,amssymb}
\usepackage{graphicx}

\font\tencmmib=cmmib10 \skewchar\tencmmib '60
\newfam\cmmibfam
\textfont\cmmibfam=\tencmmib

\def\lessim{\ \lower4pt\hbox{$
		\buildrel{\displaystyle <}\over\sim$}\ }
\def\gessim{\ \lower4pt\hbox{$\buildrel{\displaystyle >}
		\over\sim$}\ }
\def\n{\noindent}

\def\n{\noindent}

\def\Bla{\Big{\langle}}
\def\Bra{\Big{\rangle}}
\def\bla{\langle}
\def\bra{\rangle}
\def\la{\langle}
\def\ra{\rangle}

\newcommand{\e}{\mathbb{E}}

\newtheorem{lemma}{\bf Lemma}

\newtheorem{theorem}{\bf Theorem}
\newtheorem{corollary}{\bf Corollary}
\newtheorem{remark}{\bf Remark}

\newenvironment{Proof of lemma}{\noindent{\bf Proof of Lemma}}{\hfill$\Box$\newline}
\newenvironment{Proof of theorem}{\noindent{\bf Proof of Theorem}}{\hfill{\footnotesize${\square}$}\newline}
\newenvironment{Proof of theorems}{\noindent{\bf Proof of Theorems}}{\hfill$\Box$\newline}
\newenvironment{Proof of proposition}{\noindent{\bf Proof of Proposition}}{\hfill$\Box$\newline}
\newenvironment{Proof of propositions}{\noindent{\bf Proof of Propositions}}{\hfill$\Box$\newline}
\newenvironment{Proof of exercise}{\noindent{\it Proof of Exercise:}}{\hfill$\Box$}

\begin{document}
	
	\nocite{*} 
	
	\title{Parisi formula for the ground state energy\\ in the mixed $p$-spin model}
	
	\author{
				Antonio Auffinger\thanks{Department of Mathematics, Northwestern University. Email: auffing@math.northwestern.edu}
				\and
				Wei-Kuo Chen\thanks{School of Mathematics, University of Minnesota. Email: wkchen@umn.edu}
	}
	\maketitle
	
	\begin{abstract}
		
		We show that the thermodynamic limit of the ground state energy in the mixed $p$-spin model can be identified as a variational problem. This gives a natural generalization of the Parisi formula at zero temperature. 
		
	\end{abstract}

\section{Introduction and main result}

The mixed $p$-spin model is defined on the hypercube $\Sigma_N:=\{-1,+1\}^N$ for $N\geq 1$ and its Hamiltonian is given by
$$
H_N(\sigma)=X_N(\sigma)+h\sum_{i=1}^N\sigma_i,
$$
where $X_N=(X_N(\sigma):\sigma\in \Sigma_N)$ is a centered Gaussian process indexed by $\Sigma_N,$
$$
X_N(\sigma)=\sum_{p\geq 2}\frac{c_p}{N^{(p-1)/2}}\sum_{1\leq i_1,\ldots,i_p\leq N}g_{i_1,\ldots,i_p}\sigma_{i_1}\cdots\sigma_{i_p}
$$
for i.i.d. standard Gaussian random variables $g_{i_1,\ldots,i_p}$ for $1\leq i_1,\ldots,i_p\leq N$ and $p\geq 2.$ Here, $h$ denotes the strength of the external field and the sequence $(c_p)_{p\geq 2}$ is assumed to decay fast enough, for instance, $\sum_{p\geq 2}2^pc_p^2<\infty$, to guarantee the infinite sum $X_N$ converges a.s. With this assumption, one readily computes that
$$
\e X_{N}(\sigma^1)X_N(\sigma^2)=N\xi(R_{1,2}),
$$
where 
$$
\xi(s)=\sum_{p\geq 2}c_p^2s^p
$$
and $R_{1,2}=N^{-1}\sum_{i=1}^N\sigma_i^1\sigma_i^2$ is the overlap between $\sigma^1$ and $\sigma^2.$ To avoid triviality, we shall assume that $c_p\neq 0$ for at least one $p\geq 2.$ The classical Sherrington-Kirkpatrick (SK) model corresponds to $\xi(s)=s^2/2.$

A quantity of great interest in the mixed $p$-spin model is the large $N$ limit (known as the thermodynamic limit) of the ground state energy 
$$
L_N:=\max_{\sigma\in\Sigma_N}\frac{H_N(\sigma)}{N}.
$$
In the past decades, there have been several numerical studies and analytic predictions of this limit in the physics literature, especially in the case of the SK model without external field $(h=0)$ \cite{CR02, KLL, OS1, OS2, Pankov}. This quantity also obtained great relevance in problems coming from computer science (see for instance \cite{DemboMontanariSen, MMbook} and the references therein).

 In order to get an explicit expression  for 
 $$ GSE:=\lim_{N\rightarrow\infty}L_N,$$ the  usual approach is to consider the free energy
\begin{align*}
F_N(\beta)=\frac{1}{\beta N}\log \sum_{\sigma\in\Sigma_N}\exp \beta H_N(\sigma),
\end{align*}
where $\beta>0$ is called the inverse temperature. It is well-known that the thermodynamic limit  of the free energy can be computed through the famous Parisi formula. More precisely, denote by $\mathcal{M}$ the collection of all probability distribution functions on $[0,1]$. Define the Parisi functional by
\begin{align*}
\mathcal{P}_{\beta}(\alpha)=\frac{\log 2}{\beta}+\Psi_{\alpha,\beta}(0,h)-\frac{1}{2}\int_0^1\beta\alpha(s)s\xi''(s)ds
\end{align*}
for $\alpha\in\mathcal{M}$, where $\Psi_{\alpha,\beta}(t,x)$ is the weak solution to the following nonlinear parabolic PDE,
\begin{align}\label{pde}
\partial_t\Psi_{\alpha,\beta}(t,x)&=-\frac{\xi''(t)}{2}\bigl(\partial_{xx}\Psi_{\alpha,\beta}(t,x)+\beta\alpha(t)(\partial_x\Psi_{\alpha,\beta}(t,x))^2\bigr)
\end{align} 
for $(t,x)\in[0,1)\times\mathbb{R}$ with boundary condition
$$
\Psi_{\alpha,\beta}(1,x)=\frac{\log \cosh \beta x}{\beta}.
$$
For results on the regularity of $\Psi_{\alpha,\beta}$, we refer the readers to \cite{AC13,JT2}. The Parisi formula says that 
\begin{align}\label{eq0}
F(\beta)=\lim_{N\rightarrow\infty}F_N(\beta)&=\inf_{\alpha\in\mathcal{M}}\mathcal{P}_\beta(\alpha) \,\,\,\,a.s.
\end{align}
Predicted by Parisi in \cite{Parisi}, this formula was  established by Talagrand \cite{Talagrand} in the case of the mixed even $p$-spin model, i.e., $c_p=0$ for all odd $p\geq 3$. Its validity to any mixed $p$-spin model was obtained by Panchenko \cite{Panchenko}. For fixed $\beta >0$, we denote the minimizer of \eqref{eq0} by $\alpha_{P,\beta}$. Uniqueness of this minimizer was established by Auffinger-Chen \cite{AC14}.

Letting $N$ and then $\beta$ tend to infinity, the simple bound
\begin{align*}
L_N\leq F_N(\beta)\leq L_N+\frac{\log 2}{\beta}
\end{align*}
 yields that 
$$
GSE=\lim_{\beta\rightarrow\infty}F(\beta)\,\,\,\,a.s.,
$$
from which the Parisi formula deduces 
\begin{align}\label{GS}
GSE=\lim_{\beta\rightarrow\infty}\inf_{\alpha\in\mathcal{M}}\mathcal{P}_\beta(\alpha)\,\,\,\,a.s.
\end{align} 

In this paper, we show that \eqref{GS} can be expressed as a variational problem. This gives a natural generalization of Parisi's formulation to the ground state energy. To prepare for the statements of our main results, we introduce the space $\mathcal{U}$ that collects all nonnegative and nondecreasing functions $\gamma$ on $[0,1)$ that are right continuous and satisfy $\int_0^1\gamma(t)dt<\infty$. We endow this space with the $L^1$-distance $d$. Let $\mathcal{U}_d$ be the set of all step-like $\gamma\in \mathcal U$, i.e., $\gamma$ is a piecewise constant function with finite jumps. For each $\gamma\in\mathcal{U}_d,$ consider the following fully nonlinear parabolic PDE 
	\begin{align}\label{pde3}
	\partial_t{\Psi}_{\gamma}(t,x)&=-\frac{\xi''(t)}{2}\Bigl(\partial_x^2{\Psi}_{\gamma}(t,x)+\gamma(t)\bigl(\partial_x{\Psi}_{\gamma}(t,x)\bigr)^2\Bigr),
	\end{align}
	for $(t,x)\in[0,1)\times\mathbb{R}$ with boundary condition 	 $${\Psi}_\gamma(1,x)=|x|.$$ Using the Cole-Hopf transformation, $\Psi_\gamma$ can be solved explicitly in the classical sense. As we will show below that $\gamma\in(\mathcal{U}_d,d)\mapsto\Psi_\gamma(t,x)$ defines a Lipschitz functional with uniform Lipschitz constant for all $(t,x)\in[0,1]\times\mathbb{R}$, one may extend $\Psi_{\gamma}$ uniquely and continuously to arbitrary $\gamma\in\mathcal{U}.$ We shall call $\Psi_\gamma$ the Parisi PDE solution at zero temperature throughout this paper. With this construction, we can now define a continuous functional $\mathcal{P}$ on $\mathcal{U}$ by
	\begin{align*}
	\mathcal{P}(\gamma)&=\Psi_{\gamma}(0,h)-\frac{1}{2}\int_0^1t\xi''(t)\gamma(t)dt.
	\end{align*}	
    Our main result is stated as follows.
	
\begin{theorem}[Parisi formula]
	\label{thm}
	We have that
	\begin{align}
	\label{thm:eq1}
	GSE&=\inf_{\gamma\in\mathcal{U}}\mathcal{P}(\gamma)\,\,\,\,a.s.
	\end{align}
\end{theorem}

We mention that the Parisi formula for the ground state energy in the spherical version of the mixed $p$-spin model has been established recently in Chen-Sen \cite{ArnabChen15} and Jagannath-Tobasco \cite{JT}. The approaches in both works rely on the Crisanti-Sommers representation for the thermodynamic limit of the free energy, where the functional has an explicit and simple expression in terms of $\xi$ and $\alpha.$ This representation is not available in our setting, which leads to a substantially more demanding problem that requires a different approach.

The proof of Theorem \ref{thm} is based on the establishment of upper and lower inequalities between the two sides of \eqref{thm:eq1}. The upper bound is not difficult and has already been obtained by Guerra \cite[Theorem 6]{G} via choosing suitable candidates in $\mathcal{M}$ since one is taking infimum of the functional. The lower bound, in contrast, carries all the challenges as one would unavoidably need to handle the sequence $\mathcal{P}_{\beta}(\alpha_{P,\beta}),$ which involves the nonlinear PDE, 
\begin{align}\label{eq1232}
\partial_t\Psi_{\alpha_{P,\beta},\beta}(t,x)&=-\frac{\xi''(t)}{2}\bigl(\partial_{xx}\Psi_{\alpha_{P,\beta},\beta}(t,x)+\beta\alpha_{P,\beta}(t)(\partial_x\Psi_{\alpha_{P,\beta},\beta}(t,x))^2\bigr).
\end{align} 
More precisely, it is known that $q_{P,\beta}:=\inf\{t\in[0,1]:\alpha_{P,\beta}(t)=1\}<1$ and one can solve the PDE in \eqref{eq1232} for $t\in [q_{P,\beta},1]$ to get
\begin{align*}
\Psi_{\alpha_{P,\beta},\beta}(t,x)&=\frac{1}{\beta}\log \cosh \beta x+\frac{\beta}{2}\bigl(\xi'(1)-\xi'(t)\bigr)
\end{align*} 
 (see \cite[Chapter 14]{T11}).
The major obstacle here is that we do not know the quantitative behavior of $\beta\alpha_{P,\beta}(t)$ for $t$ being close to $q_{P,\beta}$ from below, when $\beta$ tends to infinity. This makes it very hard to track the effect of this singularity by a direct analysis of the PDE solution. To overcome this issue, we construct a representation of the Parisi PDE in terms of the  stochastic optimal control problem introduced in \cite{AC14,BK}. Under this framework, we are able to deal with the large $\beta$ limit of the Parisi functional and we remove this singularity as a marvel cancellation happens between the non-linear PDE  and the linear term in the functional $\mathcal{P}_\beta$. The details are in Section \ref{sec3}, where we present the argument of obtaining the lower bound.  Although we only consider the mixed $p$-spin model in this paper, we believe that the present approach could also be useful in deriving similar results as Theorem \ref{thm} from the existing Parisi formulas for the free energies in other mean-field spin glass models \cite{P05,P15,P152,P153}. 

\begin{remark} \rm We comment that one may as well formulate the Parisi functional $\mathcal{P}$ by constructing the PDE solution $\Psi_\gamma$ directly from the equation \eqref{pde3} rather than using the  above approximation procedure. However, as $\gamma$ could tend to infinity when $t$ approaches $1$ from below and the boundary condition $|x|$ is not differentiable at $0$, the construction of the PDE solution and its regularity properties require extra effort. For this reason and clarity, we use the Lipschitz property  of the PDE  (see \eqref{eq2}) to construct the functional $\mathcal{P}$.
\end{remark}

\begin{remark} \rm
Determining uniqueness of the minimizer of \eqref{thm:eq1} needs regularity properties of the solution $\Psi_\gamma$ as those used in \cite{AC14}. The proof of uniqueness in \cite{AC14} carries through once these properties are established. We do not pursue this direction here.
\end{remark}

\smallskip
\smallskip
\smallskip

{\noindent \bf Acknowledgements.} W.-K. C. thanks Giorgio Parisi for valuable suggestions and Wenqing Hu for fruitful discussions at the early stage of this work. Both authors thank the 2016 emphasis year in probability at  Northwestern University, where this work was discussed. The research of A. A. is partly supported by NSF grant DMS-1597864. The research of W.-K. C. is partly supported by NSF grant DMS-1642207 and Hong Kong Research Grants Council GRF-14302515.

\section{Variational representation for the Parisi PDE}	

In this section, we will derive a variational representation for the PDE \eqref{pde} in the form of stochastic optimal control. This formulation appeared initially in \cite{BK} and was used to establish the strict convexity of the Parisi functional $\mathcal{P}_\beta$ in \cite{AC14}. See a simplified argument of \cite{AC14} in \cite{JT2}. Different than the derivations in \cite{AC14,BK,JT2}, here we present an approach that relies only on It\^{o}'s formula. Consider the following Parisi PDE,
\begin{align}\label{pde2}
\partial_t\Psi(t,x)&=-\frac{\xi''(t)}{2}\Bigl(\partial_{x}^2\Psi(t,x)+\eta(t)\bigl(\partial_x\Psi(t,x)\bigr)^2\Bigr)
\end{align}
for $(t,x)\in[0,1)\times\mathbb{R}$ with boundary condition $\Psi(1,x)=f(x)$, where $f$ and $\eta$ are specified by one of the following two cases
\begin{itemize}
	\item[$(A1)$] $f(x)=\beta^{-1}\log \cosh \beta x$ and $\eta=\beta \alpha$ for some $\beta>0$ and $\alpha\in\mathcal{M}$,
	\item[$(A2)$] $f(x)=|x|$ and $\eta\in\mathcal{U}_d.$
\end{itemize} 
It is known \cite{AC13,JT2} that given $(A1)$, the solution $\Psi$ has the properties that $\partial_x^j\Psi\in C([0,1]\times\mathbb{R})$ for all $j\geq 0$ and $|\partial_x\Psi|\leq 1.$
Likewise, since $\gamma\in\mathcal{U}_d$ is a step function, $\Psi$ can be solved via the Cole-Hopf transformation, from which it can be checked that $\partial_x^j\Psi\in C([0,1)\times\mathbb{R})$ for all $j\geq 0$ and $|\partial_x\Psi|\leq 1.$ 

Let $W=(W_t)_{0\leq t\leq 1}$ be a standard Brownian motion. For $0\leq s<t\leq 1$, denote by $D[s,t]$ the collection of all progressive measurable processes $u$ on $[0,1]$ with respect to the filtration generated by $W$ and satisfying $\sup_{0\leq t\leq 1}|u(t)|\leq 1.$ For any $x\in\mathbb{R}$ and $u\in D[s,t],$ define
\begin{align*}
F^{s,t}(u,x)=\e\left[C^{s,t}(u,x)-L^{s,t}(u)\right],
\end{align*}
where
\begin{align*}
C^{s,t}(u,x)&=\Psi\left(t,x+\int_s^t\eta(r)\xi''(r)u(r)dr+\int_s^t\xi''(r)^{1/2}dW_r\right),\\
L^{s,t}(u)&=\frac{1}{2}\int_s^t\eta(r)\xi''(r)\e u(r)^2dr.
\end{align*} 
Note that these functionals are well-defined as $\int_{0}^{1}\eta (r) dr <\infty$ and $|u(r)| \leq 1$ for all $r \in [0,1]$.

\begin{theorem}[Variational formula] \label{thm0}
	Let $f$ and $\eta$ satisfy $(A1)$ or $(A2).$  
We have that 
        \begin{equation*}
		\Psi(s,x)=\max\left\{F^{s,t}(u,x)|u\in D[s,t]\right\}.
		\end{equation*}

\end{theorem}

\begin{proof}
	For simplicity, we shall only consider the case that $\eta$ is continuous on $[0,1]$. Under this assumption, the PDE \eqref{pde2} is valid in the classical sense and this allows us to use It\^{o}'s formula. The general case can be treated by an approximation argument identical to \cite[Theorem 3]{AC14}.
	Let $u\in D[s,t].$ For notational convenience, we denote $$Y(r)=x+\int_s^r\eta(w)\xi''(w)u(w)dw+\int_s^r\xi''(w)^{1/2}dW_w.$$
	Define 
	\begin{align*}
	Z(r)&=\Psi(r,Y(r))+\frac{1}{2}\int_s^r\eta(w)\xi''(w)\bigl(\partial_x\Psi(w,Y(w))-u(w)\bigr)^2dw\\
	&-\int_s^r\xi''(w)^{1/2}\partial_x\Psi(w,Y(w))dW_w-\frac{1}{2}\int_s^r\eta(w)\xi''(w)u(w)^2dw.
	\end{align*} 
	Using It\^{o}'s formula, we obtain that
	\begin{align*}
	&\Psi(w,Y(w))=\partial_w\Psi(w,Y(w))dw+\partial_x\Psi(w,Y(w))dY(w)+\frac{1}{2}\partial_{xx}\Psi(w,Y(w))d\la Y, Y \ra_w.
	\end{align*}
	Here, from the PDE \eqref{pde2}, the right-hand side becomes
	\begin{align*}
	&-\frac{\xi''}{2}\bigl(\partial_w^2\Psi(w,Y)+\eta(w)\bigl(\partial_x\Psi(w,Y)\bigr)^2\bigr)\\
	&+\eta \xi'' u \partial_x\Psi(w,Y )dw+{\xi''}^{1/2}\partial_x\Psi(w,Y )dW_w+\frac{\xi''}{2}\partial_{xx}\Psi(w,Y )dw\\
	&=-\frac{1}{2}\eta \xi'' \bigl(\bigl(\partial_x\Psi(w,Y )\bigr)^2-2u \partial_x\Psi(w,Y )\bigr)dw+{\xi''}^{1/2}\partial_x\Psi(w,Y )dW_w\\
	&=-\frac{1}{2}\eta \xi'' \bigl(\partial_x\Psi(w,Y )-u \bigr)^2dw+{\xi''}^{1/2}\partial_x\Psi(w,Y )dW_w+\frac{1}{2}\eta \xi'' u^2dw.
	\end{align*}
	In other words, $dZ(r)=0$ and this implies that $Z(t)=Z(s)=\Psi(s,x)$. Taking expectation of this equation gives 
	\begin{align}
	\begin{split}\label{lem:proof:eq1}
	\Psi(s,x)&=\e \Psi\big(t,Y(t)\bigr)-\frac{1}{2}\int_s^t\eta(r)\xi''(r)\e u(r)^2dr\\
	&+\frac{1}{2}\int_s^t\eta(r)\xi''(r) \e\bigl(\partial_x\Psi(r,Y(r))-u(r)\bigr)^2dr,
	\end{split}
	\end{align}
	so
	\begin{align}
	\begin{split}\label{lem:proof:eq2}
	\Psi(s,x)&\geq \sup_{u\in D[s,t]}\Bigl(\e \Psi\Big(t,x+\int_s^t\eta(r)\xi''(r)u(r)dr+\int_s^t\xi''(r)^{1/2}dW_r\Bigr)\\
	&\qquad\qquad\qquad\qquad\qquad-\frac{1}{2}\int_s^t\eta(r)\xi''(r)\e u(r)^2dr\Bigr).
	\end{split}
	\end{align}
To obtain the optimality, in the case of either $(A1)$ or $(A2)$ with $t<1$, we consider 
     \begin{align}\label{max}
     u^*(r)&=\partial_x\Psi(r,X(r)),
     \end{align}
     where $(X(r))_{s\leq r\leq t}$ is the strong solution to
     \begin{align*}
     dX(r)&=\eta(r)\xi''(r)\partial_x\Psi(r,X(r))dr+\xi''(r)^{1/2}dW(r),\\
     X(s)&=x.
     \end{align*}
Simply notice that if $u=u^*$, then $Y(r)=X(r)$  for $s\leq r\leq t$ such that $\partial_x\Psi(r,Y(r))=u^*$ for all $s\leq r\leq t$. From this and  \eqref{lem:proof:eq1}, the equality of \eqref{lem:proof:eq2} follows. For the case $(A2)$ with $t=1$, we take $u^*$ to be the same as \eqref{max} for $s\leq r<1$ and set $u^*(1)=0$. Letting $u=u^*$, one sees that $Y(r)=X(r)$ for $s\leq r<1$ such that $\partial_x\Psi(r,Y(r))=u^*$ for all $s\leq r <1$. The optimality remains true. Note that the reason why one could not take the same $u^*$ directly from \eqref{max} on the whole $[s,1]$ is because $\Psi(1,x)=|x|$ is not differentiable at $0$. 
\end{proof}

\begin{remark}
	\rm In \cite{AC14,BK,JT2}, the fact that the process $u^*$ attains the maximum value of $F^{s,t}(\cdot,x)$ was established by a direct verification using It\^{o}'s formula. From the above proof, the equation \eqref{lem:proof:eq1} quantifies the distance between the PDE and the functional $F^{s,t}$ for any $u.$ Furthermore, it also explains how to choose the right candidate to reach the optimality.
\end{remark}

As an immediate consequence of Theorem \ref{thm0},  we obtain the variational representations for the PDE's $\Psi_{\alpha,\beta}$ for $\alpha\in\mathcal{M}$ and $\Psi_\gamma$ for $\gamma\in\mathcal{U}_d.$

\begin{corollary}
For any $\alpha\in\mathcal{M}$, we have that
\begin{align}
\begin{split}\label{cor1:eq1}
\Psi_{\alpha,\beta}(0,x)&=\sup_{u\in D[0,1]}\Bigl(\frac{1}{\beta}\e \log \cosh \beta\Bigl(x+\int_{0}^1 u(r)\xi''(r)\beta\alpha(r)dr+\int_0^1\xi''(r)^{1/2}dW_r\Bigr)\\
&\qquad\qquad\qquad-\frac{1}{2}\int_{0}^1\e  u(r)^2\xi''(r)\beta\alpha(r)dr\Bigr).
\end{split}
\end{align}
\end{corollary}
\begin{corollary}\label{cor2}
	Let $s\in[0,1]$ and $x\in\mathbb{R}.$ For any $\gamma,\gamma'\in\mathcal{U}_d$, we have that
		\begin{align}
		\begin{split}\label{eq1}
		\Psi_{\gamma}(s,x)&=\sup_{u\in D[s,1]}\Bigl(\e \Bigl|x+\int_{s}^1 u(s)\xi''(r)\gamma(r)dr+\int_s^1\xi''(r)^{1/2}dW_r\Bigr|\\
		&\qquad\qquad\qquad-\frac{1}{2}\int_{s}^1\e  u(r)^2\xi''(r)\gamma(r)dr\Bigr)
		\end{split}
		\end{align}
		and
		\begin{align}
		\label{eq2}
				|\Psi_\gamma(s,x)-\Psi_{\gamma'}(s,x)|\leq 2\xi''(1)d(\gamma,\gamma').
		\end{align}
\end{corollary}

Here, \eqref{eq2} can be checked directly from \eqref{eq1} by noting that $|x|$ is Lipschitz $1$ and using the triangle inequality. With the Lipschitz property \eqref{eq2}, one can now extend the solution $\Psi_\gamma$ continuously and uniquely to all $\gamma\in\mathcal{U}$ by an approximation procedure using $\mathcal{U}_d.$ It is then clear that both \eqref{eq1} and \eqref{eq2} are valid for any $\gamma\in\mathcal{U}.$

\section{Proof of Parisi's formula}\label{sec3}

First we construct a weakly convergent subsequence of $(\beta \alpha_{P,\beta})_{\beta>0}$ as follows. Note that from Gaussian integration by parts, one has the following identity,
\begin{align}\label{sec3.1:eq1}
\beta\Bigl(\xi(1)-\e\bla \xi(R_{1,2})\bra_\beta\Bigr)=\e \Bla\frac{X_N(\sigma)}{N}\Bra_\beta,
\end{align}
where letting 
\begin{align*}
G_{N,\beta}(\sigma)=\frac{\exp \beta H_N(\sigma)}{\sum_{\sigma\in\Sigma_N}\exp \beta H_N(\sigma)}
\end{align*}
be the Gibbs measure, $\sigma^1,\sigma^2$ are two i.i.d. samplings from $G_{N,\beta}$ and $\la \cdot\ra_\beta$ is the Gibbs average with respect to this measure.
To control \eqref{sec3.1:eq1}, we observe that
\begin{align*}
\e \Bla\frac{X_N(\sigma)}{N}\Bra_\beta&\leq \e\max_{\sigma\in \Sigma_N}\frac{X_N(\sigma)}{N}\leq \sqrt{2\xi'(1)\log 2}.
\end{align*}
Here the second inequality is obtained by using the usual estimate for the size of the maximum of Gaussian process (see e.g.  \cite{T03}). It is well-known (see  \cite{P08}) that $\beta F(\beta)$ is differentiable in temperature, which yields
\begin{align*}
\lim_{N\rightarrow\infty}\e\bla \xi(R_{1,2})\bra_\beta&=\int_0^1\xi(s)\alpha_{P,\beta}(ds).
\end{align*}
Consequently, from \eqref{sec3.1:eq1} and integration by part,
\begin{align*}
\int_0^1\beta \alpha_{P,\beta}(s)\xi'(s)ds=\beta\Bigl(\xi(1)-\int_0^1\xi(s)\alpha_{P,\beta}(ds)\Bigr)\leq \sqrt{2\xi'(1)\log 2}.
\end{align*}
From this inequality, since $\alpha_{P,\beta}$ is nondecreasing, it follows that
\begin{align}\label{eq-1}
\beta \alpha_{P,\beta}(s)\leq \frac{\sqrt{2\xi'(1)\log 2}}{\xi(1)-\xi(s)},\,\,\forall s\in[0,1).
\end{align}
From the last two inequalities, we may use Helly's selection theorem combined with a diagonalization process to conclude that, without loss of generality, 
\begin{align}\label{eq-3}
\gamma_0:=\lim_{\beta\rightarrow\infty}\beta \alpha_{P,\beta}
\end{align}  
exists weakly on $[0,1)$ and 
\begin{align}
\label{eq-4} 
L_0:=\lim_{\beta\rightarrow\infty}\int_0^1\beta\alpha_{P,\beta}(s)\xi''(s)ds
\end{align}
exists. The following lemma, though simple, will be of great use in our argument.

\begin{lemma}\label{lem3}
	Let $(\alpha_\beta)_{\beta>0}\in\mathcal{M}$ such that $(\beta\alpha_\beta)$ converges to $\gamma  $ weakly on $[0,1)$ for some $\gamma\in\mathcal{U}$ and $$
	\int_0^1\beta\alpha_\beta(s)\xi''(s)ds\rightarrow L.$$ If $\phi$ is any measurable function with $\|\phi\|_\infty\leq 1$ and $\lim_{t\rightarrow 1-}\phi(t)=\phi(1)$ a.s., then
	\begin{align*}
	\lim_{\beta\rightarrow\infty}\int_0  ^1\beta\alpha_\beta(s)\xi''(s)\phi(s)ds&=\int_0  ^1\xi''(s)\phi(s)\nu  (ds),
	\end{align*}  
	where $\nu  $ is the measure induced by $$
	\nu  (ds)=1_{[0,1)}(s)\gamma  (s)ds+\frac{1}{\xi''(1)}\Bigl(L  -\int_0  ^1\gamma  (s)\xi''(s)ds\Bigr)\delta_1(ds)
	$$
	for $\delta_1(ds)$ the Dirac measure at $1.$
\end{lemma}

\begin{proof}
	Write
	\begin{align*}
	\int_0^1\beta\alpha_\beta(s)\xi''(s) \phi(s)ds&=\int_0^t\beta\alpha_\beta(s)\xi''(s) \phi(s)ds+\int_t^1\beta\alpha_\beta(s)\xi''(s) \phi(s)ds.
	\end{align*}
	Observe that
	\begin{align*}
	&\Bigl|\int_t^1\beta\alpha_\beta(s)\xi''(s) \phi(s)ds- \phi(1)\Bigl(\int_0^1\beta\alpha_\beta(s)\xi''(s)ds-\int_0^t\beta\alpha_\beta(s)\xi''(s)ds\Bigr)\Bigr|\\
	&=\Bigl|\int_t^1\beta\alpha_\beta(s)\xi''(s)( \phi(s)- \phi(1))ds\Bigr|\\
	&\leq \int_t  ^1\beta\alpha_\beta(s)\xi''(s)ds\max_{t\leq s\leq 1}| \phi(s)- \phi(1)|.
	\end{align*}
	From this, it follows that by \eqref{eq-1} and the dominated convergence theorem,
	\begin{align*}
	\limsup_{\beta\rightarrow\infty}\int_t^1\beta\alpha_\beta(s)\xi''(s) \phi(s)ds&\leq  \phi(1)\Bigl(L-\int_0^t\xi''(s)\n\phi(ds)\Bigr)+L  \max_{t\leq s\leq 1}| \phi(s)- \phi(1)|,\\
	\liminf_{\beta\rightarrow\infty}\int_t^1\beta\alpha_\beta(s)\xi''(s) \phi(s)ds&\geq  \phi(1)\Bigl(L-\int_0^t\xi''(s)\n\phi(ds)\Bigr)-L  \max_{t\leq s\leq 1}| \phi(s)- \phi(1)|.
	\end{align*}
	On the other hand, using \eqref{eq-1} and the dominated convergence theorem again,
	\begin{align*}
	&\limsup_{\beta\rightarrow\infty}\int_0  ^1\beta\alpha_\beta(s)\xi''(s) \phi(s)ds\\
	&\leq \limsup_{\beta\rightarrow\infty}\int_0  ^t\beta\alpha_\beta(s)\xi''(s) \phi(s)ds+\limsup_{\beta\rightarrow\infty}\int_t^1\beta\alpha_\beta(s)\xi''(s) \phi(s)ds\\
	&\leq\int_0  ^t \phi(s)\gamma(s)\xi''(s)ds+  \phi(1)\Bigl(L-\int_0^t\xi''(s)\n\phi(ds)\Bigr)+L  \max_{t\leq s\leq 1}| \phi(s)- \phi(1)|
	\end{align*} 
	and
	\begin{align*}
	&\liminf_{\beta\rightarrow\infty}\int_0  ^1\beta\alpha_\beta(s)\xi''(s) \phi(s)ds\\
	&\geq \liminf_{\beta\rightarrow\infty}\int_0  ^t\beta\alpha_\beta(s)\xi''(s) \phi(s)ds+\liminf_{\beta\rightarrow\infty}\int_t^1\beta\alpha_\beta(s)\xi''(s) \phi(s)ds\\
	&\geq \int_0^t \phi(s)\gamma(s)\xi''(s)ds+  \phi(1)\Bigl(L-\int_0^t\xi''(s)\n\phi(ds)\Bigr)-L  \max_{t\leq s\leq 1}| \phi(s)- \phi(1)|.
	\end{align*} 
	Since these two inequalities hold for all $t\in(0,1),$ letting $t\rightarrow 1-$, the continuity of $\phi$ at $1$ implies the announced result.
\end{proof}

The proof of Theorem \ref{thm} is completed by the following two lemmas, upper and lower bounds. We deal with the upper bound first. It agrees with the one that appeared in  \cite[Theorem 6]{G}.

\begin{lemma}[Upper bound] \label{up}
 We have that
	\begin{align*}
	GSE\leq \inf_{\gamma\in\mathcal{U}}\Bigl(\Psi_\gamma(0,h)-\frac{1}{2}\int_0^1\gamma(s)s\xi''(s)ds\Bigr).
	\end{align*}
\end{lemma}

\begin{proof}
	Assume that $\gamma\in \mathcal{U}$ with $\gamma(1-)<\infty.$ For $\beta>\gamma(1-),$ define 
	\begin{align*}
	\alpha_\beta(s)&=\frac{\gamma(s)}{\beta}1_{[0,1)}(s)+\delta_1(s).
	\end{align*}
	Since $\gamma$ is nonnegative and nondecreasing with right continuity, we see that $\alpha_\beta\in\mathcal{M}.$ From the PDE \eqref{pde}, $\Psi_{\alpha_\beta,\beta}$ is given by
	\begin{align*}
	\partial_t{\Psi}_{\alpha_{\beta},\beta}(t,x)&=-\frac{\xi''(t)}{2}\Bigl(\partial_x^2{\Psi}_{\alpha_\beta,\beta}(t,x)+\gamma(t)\bigl(\partial_x{\Psi}_{\alpha_\beta,\beta}(t,x)\bigr)^2\Bigr)
	\end{align*}
	for $(t,x)\in[0,1)\times\mathbb{R}$ with boundary condition ${\Psi}_{\gamma,\beta}(1,x)=\beta^{-1}\log\cosh \beta x.$ In other words, $\Psi_{\alpha_\beta,\beta}$ follows the same PDE. As now the boundary condition satisfies $\lim_{\beta\rightarrow\infty}\beta^{-1}\log\cosh \beta x=|x|$, it can be seen that $\lim_{\beta\rightarrow\infty}\Psi_{\alpha_\beta,\beta}=\Psi_\gamma,$ where $\Psi_\gamma$ is the solution to \eqref{pde3}. On the other hand, note that
	\begin{align*}
	\int_0^1\beta\alpha_{\beta}(s)s\xi''(s)ds&=\int_0^{1}s\xi''(s)\gamma(s)ds.
	\end{align*}
	Thus, we can conclude that
	\begin{align*}
	\lim_{\beta\rightarrow\infty} \mathcal{P}_{\beta}(\alpha_\beta)&=\lim_{\beta\rightarrow\infty}\Psi_{\alpha_\beta,\beta}(0,h)-\lim_{\beta\rightarrow\infty}\int_0^1\beta\alpha_{\beta}(s)s\xi''(s)ds\\
	&=\Psi_\gamma(0,h)-\frac{1}{2}\int_0^1s\xi''(s)\gamma(s)ds
	\end{align*}
	and so
	\begin{align}
	\begin{split}\label{eq6}
	\Psi_\gamma(0,h)-\frac{1}{2}\int_0^1s\xi''(s)\gamma(s)ds=\lim_{\beta\rightarrow\infty}\mathcal{P}_{\beta}(\alpha_\beta)&\geq \lim_{\beta\rightarrow\infty}\inf_{\alpha\in\mathcal{M}}\mathcal{P}_\beta(\alpha)=GSE.
	\end{split}
	\end{align}
	To establish the same inequality for any $\gamma\in\mathcal{U}$ without the assumption $\gamma(1-)<\infty,$ we may apply the truncation $\gamma_n=\min(\gamma,n)$ of $\gamma$ to \eqref{eq6},
	\begin{align*}
	\Psi_{\gamma_n}(0,h)-\frac{1}{2}\int_0^1s\xi''(s)\gamma_n(s)ds\geq GSE.
	\end{align*}
	Since $(\gamma_n)_{n\geq 1}$ converges to $\gamma$ under the distance $d$, using \eqref{eq2} for $\gamma_n$ and $\gamma$ leads to \eqref{eq6} for any $\gamma\in\mathcal{U}.$ This finishes our proof.
\end{proof}

    Next, we establish the lower bound, which is the most critical part in our approach.
    
\begin{lemma}[Lower bound]\label{lp}
	We have that
	\begin{align*}
	GSE&\geq \inf_{\gamma\in\mathcal{U}}\Bigl(\Psi_{\gamma}(0,h)-\frac{1}{2}\int_0^1\xi''(s)s\gamma(s)ds\Bigr).
	\end{align*}
\end{lemma}

\begin{proof}
	Recall $\gamma_0$ and $L_0$ from \eqref{eq-3} and \eqref{eq-4}.
	Define $$
	\nu_0(ds)=\gamma_0(s)1_{[0,1)}(s)ds+\xi''(1)^{-1}\Bigl(L_0-\int_0^1\gamma_0(s)\xi''(s)ds\Bigr)\delta_1(ds).
	$$
	For each $n\geq 1,$ consider the function $g_n$ defined by
	\begin{align*}
	g_n(x)&=\left\{
	\begin{array}{ll}
	1,&\mbox{if $x\geq 0$},\\
	2nx+1,&\mbox{if $-n^{-1}\leq x<0$},\\
	-1,&\mbox{if $x<-n^{-1}$.}
	\end{array}
	\right.
	\end{align*}
	Note that $(g_n)_{n\geq 1}$ is a sequence of continuous functions with $\|g_n\|_\infty\leq 1$ such that 
	$$
	\lim_{n\rightarrow \infty}g_n(x)=\mbox{sign}(x),
	$$
	where
	 $$
	 \mbox{sign}(x):=\left\{
	 \begin{array}{ll}
	 1,&\mbox{if $x\geq 0$},\\
	 -1,&\mbox{if $x<0$}.
	 \end{array}
	 \right.
	 $$
	Let $u\in D[0,1]$. For any $\varepsilon\in(0,1)$ and $n\geq 1,$ define
	\begin{align*}
	\phi_{\varepsilon,n}(s)&=u(s)1_{[0,\varepsilon)}(s)+g_n\Bigl(h+\int_0^su(r)\xi''(r)\gamma_0(r)dr+\int_0^s\xi''(r)^{1/2}dW_r\Bigr)1_{[\varepsilon,1]}(s).
	\end{align*}
	Then $\phi_{\varepsilon,n}\in D[0,1]$ and $\lim_{s\rightarrow 1-}\phi_{\varepsilon,n}(s)=\phi_{\varepsilon,n}(1)$ since $g_n$ is continuous on $\mathbb{R}$. In addition, the following limits hold,
	\begin{align*}
	\phi_{n}(s):=\lim_{\varepsilon\rightarrow 1-}\phi_{\varepsilon,n}(s)&=
	u(s)1_{[0,1)}(s)+g_n(S)1_{\{1\}}(s)
	\end{align*}
	and
	\begin{align*}
			\lim_{n\rightarrow\infty}\phi_n(s)&=u(s)1_{[0,1)}(s)+\mbox{sign}(S)1_{\{1\}}(s),
	\end{align*}
	where 
	$$
	S:=h+\int_{0}^1 u(s)\xi''(s)\gamma_0(s)ds+\int_0^1\xi''(s)^{1/2}dW_s.
	 $$
	Now, using \eqref{GS} and the formula \eqref{cor1:eq1} applied to $\Psi_{\alpha_{P,\beta},\beta}(0,h)$, the Fatou lemma and Lemma \ref{lem3} together imply
	\begin{align*}
	GSE&=\lim_{\beta\rightarrow\infty}\mathcal{P}_\beta(\alpha_{P,\beta})\\
	&=\lim_{\beta\rightarrow\infty}\Bigl(\frac{\log 2}{\beta}+\Psi_{\alpha_{P,\beta},\beta}(0,h)-\frac{1}{2}\int_0^1\beta\alpha_{P,\beta}(s)\xi''(s)sds\Bigr)\\
	&\geq \lim_{\beta\rightarrow\infty}\Big(\frac{1}{\beta}\e \log \cosh \beta\Bigl(h+\int_0^1\beta\alpha_{P,\beta}(s)\xi''(s)\phi_{\varepsilon,n}(s)ds+\int_0^1\xi''(s)^{1/2}dW_s\Bigr)\\
	&\qquad \qquad\,\,-\frac{1}{2}\int_0^1\beta\alpha_{P,\beta}(s)\xi''(s)(\e \phi_{\varepsilon,n}(s)^2+s)ds\Bigr)\\
	&\geq \e \Bigl|h+\int_0^1 \phi_{\varepsilon,n}(s)\xi''(s)\nu_0(ds)+\int_0^1\xi''(s)^{1/2}dW_s\Bigr|-\frac{1}{2}\int_0^1(\e  \phi_{\varepsilon,n}(s)^2+s)\xi''(s)\nu_0(ds).
	\end{align*}
	By the dominated convergence theorem, letting $\varepsilon\rightarrow 1-$ and then $n\rightarrow\infty$ implies
	\begin{align*}
	GSE&\geq \e \bigl|S+\mbox{sign}(S)\xi''(1)\nu_0(1)\bigr|\\
	&\quad-\frac{1}{2}\int_0^1(\e  u(s)^2+s)\xi''(s)\gamma_0(s)ds-\frac{1}{2}\bigl(\e \bigl(\mbox{sign}(S)\bigr)^2+1\bigr)\xi''(1)\nu_0(1).
	\end{align*} 
    Since
	\begin{align}
	\begin{split}
	\label{ob}
    &\bigl|S+\mbox{sign}(S)\xi''(1)\nu_0(1)\bigr|\\
	&=\bigl(S+\xi''(1)\nu_0(1)\bigr)1_{\{S> 0\}}-\bigl(S-\xi''(1)\nu_0(1)\bigr)1_{\{S< 0\}}+\xi''(1)\nu_0(1)1_{\{S=0\}}\\
	&=|S|+\xi''(1)\nu_0(1)
		\end{split}
	\end{align}
	and
		\begin{align}\label{ob2}
		\bigl(\e \bigl(\mbox{sign}(S)\bigr)^2+1\bigr)\xi''(1)\nu_0(1)&=2\xi''(1)\nu_0(1),
		\end{align}
		it follows that the terms $\xi''(1)\nu_0(1)$ cancel each other and we obtain
		\begin{align*}
		GSE&\geq \e \bigl|S\bigr|-\frac{1}{2}\int_0^1(\e  u(s)^2+s)\xi''(s)\gamma_0(s)ds\\
		&=\e\Bigl| h+\int_{0}^1 u(s)\xi''(s)\gamma_0(s)ds+\int_0^1\xi''(s)^{1/2}dW_s\Bigr|\\
		&\quad-\frac{1}{2}\int_0^1(\e  u(s)^2+s)\xi''(s)\gamma_0(s)ds.
		\end{align*}
	Since this holds for all $u\in D[0,1]$, 
	we get that by noting \eqref{eq1} holds for any $\gamma\in\mathcal{U}$,
	\begin{align*}
	GSE&\geq \Psi_{\gamma_0}(0,h)-\frac{1}{2}\int_0^1s\xi''(s)\gamma_0(s)ds\\
	&\geq \inf_{\gamma\in\mathcal{U}}\Bigl(\Psi_{\gamma}(0,h)-\frac{1}{2}\int_0^1s\xi''(s)\gamma(s)ds\Bigr).
	\end{align*}
	This finishes our proof.
\end{proof}

\begin{remark}
	\rm From the construction of $\phi_{\varepsilon,n}$, one sees that the crucial observations \eqref{ob} and \eqref{ob2} allow us to cancel out the common terms $\xi''(1)\nu_0(1)$ arising from the jump of $\nu_0$ at $1.$ This explains how the effect of the singularity of the Parisi PDE near $1$ in the limiting procedure is eliminated by the linear term of the Parisi functional as mentioned in the introduction.
\end{remark}
	
\begin{remark}\rm
It should be mentioned that while the expression \eqref{thm:eq1} depends only on $\gamma,$ the Parisi formula for the ground state energy in the spherical mixed $p$-spin model \cite{ArnabChen15,JT} relies on $\gamma$ and one extra variable, $L$, playing the role like $L_0$ in \eqref{eq-4}. 
\end{remark}

\begin{proof}[\bf Proof of Theorem \ref{thm}]
	The equality \eqref{thm:eq1} follows directly from Lemmas \ref{up} and \ref{lp}. 
\end{proof}

\end{document}